\documentclass[10pt, reqno]{amsart}

\usepackage{amsmath}
\usepackage{amscd}
\usepackage{amsthm}
\usepackage{amssymb}
\usepackage{latexsym}
\usepackage{euscript}
\usepackage{epsfig}
\usepackage{graphics}
\usepackage{array}
\usepackage{enumerate}
\usepackage{dsfont}
\usepackage{color}
\usepackage{wasysym}
\usepackage{graphicx}
\usepackage{amsfonts}
\usepackage{mathrsfs}
\usepackage{dsfont}
\usepackage{indentfirst}
\usepackage{esint}

\usepackage[unicode=true,
 bookmarks=true,bookmarksnumbered=false,bookmarksopen=false,
 breaklinks=false,pdfborder={0 0 1},backref=false,colorlinks=true]{hyperref}
\hypersetup{pdftitle={Harmonic functions on metric measure spaces},
 pdfauthor={Bobo Hua, Martin Kell and Chao Xia}}

\def\Li{\mathrm{Lip}}

\newcommand{\R}{{\mathbb R}}

\newcommand{\supp}{{\mathrm{supp}\,}}

\newcommand{\ess}{\mathrm{ess}}
\newcommand{\loc}{\mathrm{loc}}
\newcommand{\Ch}{\operatorname{Ch}}

\newcommand{\Lip}{\operatorname{Lip}}

\newcommand{\cutoff}{\mathsf{Cutoff}}

\newtheorem{thm}{Theorem}[section]
\newtheorem{coro}[thm]{Corollary}
\newtheorem{cor}[thm]{Corollary}
\newtheorem{lemma}[thm]{Lemma}
\newtheorem{lem}[thm]{Lemma}
\newtheorem{pro}[thm]{Proposition}
\newtheorem{prop}[thm]{Proposition}

\theoremstyle{definition}

\newtheorem{defn}[thm]{Definition}

\newtheorem{rem}[thm]{Remark}
\newtheorem{rem*}[thm]{Remark}

\newcommand{\Hmm}[1]{\leavevmode{\marginpar{\tiny%
$\hbox to 0mm{\hspace*{-0.5mm}$\leftarrow$\hss}%
\vcenter{\vrule depth 0.1mm height 0.1mm width \the\marginparwidth}%
\hbox to 0mm{\hss$\rightarrow$\hspace*{-0.5mm}}$\\\relax\raggedright
#1}}}

\begin{document}

\title[Harmonic functions on metric measure spaces] 
{Harmonic functions on metric measure spaces} 
\author{Bobo Hua}
\address{School of Mathematical Sciences, LMNS, Fudan University, Shanghai 200433, China; Shanghai Center for Mathematical Sciences, Fudan University, Shanghai 200433, China}
\email{bobohua@fudan.edu.cn}
\author{Martin Kell}
\address{Mathematisches Institut, Universit\"at T\"ubingen, Auf der Morgenstelle
10, 72076 T\"ubingen, Germany}
\email{martin.kell@math.uni-tuebingen.de}
\author{Chao Xia}
\address{School of Mathematical Sciences,
Xiamen University,  Xiamen, 
361005, P.R. China}
\email{chaoxia@xmu.edu.cn}


\thanks{The research leading to these results has
received funding from the European Research Council under the
European Union's Seventh Framework Programme (FP7/2007-2013) / ERC
grant agreement n$^\circ$ 267087. M.K. was supported by the IMPRS
``Mathematics in the Sciences''}

\begin{abstract}
In this paper,  we study harmonic functions on metric measure spaces
with Riemannian Ricci curvature bounded from below, which were introduced by
Ambrosio-Gigli-Savar\'e. We prove a Cheng-Yau type local gradient estimate for
harmonic functions on these spaces. Furthermore, we derive various optimal
dimension estimates for spaces of polynomial growth harmonic functions on
metric measure spaces with nonnegative Riemannian Ricci curvature.
\end{abstract}
\maketitle

\section{Introduction}

In \cite{BakryEmery85} Bakry and \'Emery introduced  the so-called
$\Gamma$-calculus and a purely analytical \textit{Curvature-Dimension
condition} $BE(K,N), N \in [1,\infty]$  for Riemannian manifolds,
which are applicable to the general setting of Dirichlet forms and
the associated Markov semigroups.
Some years later
Lott-Villani \cite{LottVillani09} and Sturm \cite{Sturm06a,Sturm06b}
introduced independently another \textit{Curvature-Dimension condition} $CD(K,N), N \in [1,\infty]$  for general metric measure spaces (mms for
short) coming from a better understanding
of gradient flows on associated Wasserstein spaces. To overcome the lack of a local-to-global property for $CD(K,N)$ with finite $N$,
Bacher-Sturm \cite{BS10} introduced a weaker notion called \textit{reduced Curvature-Dimension condition} $CD^*(K,N), N\in [1,\infty)$.  The two notions $BE(K,N)$ and $CD(K,N)$ are both equivalent
to the condition that for weighted Riemannian manifolds the weighted $N$-Ricci curvature has lower bound $K$.

Most recently, Ambrosio-Gigli-Savar\'e made a breakthrough in series
of fundamental papers \cite{AGS2011, AGS2011a, AmbrosioGigliMondinoRajala12},
by showing that
the two notions $CD(K,\infty)$ and $BE(K,\infty)$ 
are equivalent for \textit{infinitesimal Hilbertian} mms. They gave a new notion $RCD(K,\infty)$ to indicate such spaces and called them \textit{mms with Riemannian Ricci curvature bounded from below}.  Actually these spaces exclude Finsler manifolds since the infinitesimal Hilbertian property implies that the heat flow is linear. 
Later Erbar-Kuwada-Sturm \cite{ErbarKuwadaSturm13}  introduced for
finite dimensional constant $N$ the class  $RCD^*(K,N)$  and they
established the equivalence between $CD^*(K,N)$ and $BE(K,N)$ for
infinitesimal Hilbertian mms. This was also
independently discovered  by
Ambrosio-Mondino-Savar\'e \cite{AMS13}.


The main goal of this paper is to study harmonic functions on
$RCD^*(K,N)$ mms for \textit{finite} $N$ (see Definition \ref{RCD}). We refer to  \cite{ErbarKuwadaSturm13} for various equivalent definitions of $RCD^*(K,N)$. With the calculus developed in recent years, one may
expect that many results in smooth Riemannian manifolds and
Alexandrov spaces (see
\cite{BuragoGromovPerelman92,BuragoBuragoIvanov01} for definitions)
can be extended to $RCD^*(K,N)$ mms.  We will show in this
paper that a local calculus can be developed in order to prove local gradient estimates for harmonic functions,
as well as 
dimension estimates for spaces of polynomial growth harmonic
functions.

Let us start with a brief introduction of the framework. Throughout
the paper,  we assume
$$\begin{array}{lll}&&(X, d, m)\hbox{  is a metric measure space,} \\
&& \hbox{ where }(X, d)\hbox{ is a complete and separable metric space, }\\
&&\hbox{ and }m\hbox{ is a nonnegative }\sigma\hbox{-finite Borel measure.}\end{array}$$
Note that we do not require $X$ to be compact. We make the setup for
$\sigma$-finite Borel measures in order to
include important geometric objects such as noncompact finite
dimensional Riemannian manifolds with Ricci lower bounds equipped
with the volume measure, the measured Gromov-Hausdorff limit spaces
of Riemannian manifolds with a uniform Ricci lower bound and a
uniform dimension upper bound equipped with a natural Radon measure,
see Cheeger-Colding
\cite{CheegerColding97,CheegerColding00II,CheegerColding00III} and
finite dimensional Alexandrov spaces with Ricci lower bounds, see
Zhang-Zhu \cite{ZhangZhu10} and Petrunin \cite{Petrunin11}.

\

The first part of the paper is devoted to a quantitative gradient
estimate of harmonic functions. In 1975, Yau \cite{Yau75} proved a
Liouville type theorem for harmonic functions on Riemannian manifolds
with nonnegative Ricci curvature. Then Cheng-Yau \cite{ChengYau75}
used Bochner's technique to derive a local gradient estimate for
harmonic functions, which is now a fundamental result in geometric
analysis. On general metric spaces,  Bochner's technique fails due to the lack of higher differentiability of the metric. 
By the semigroup approach, Garofalo-Mondino
\cite{Garofalo2013} proved Li-Yau type global gradient estimate for
solutions of heat equations on $RCD^*(K,N)$ mms equipped with a probability
measure. Their arguments heavily rely on the probability measure
assumption. A generalization to the $\sigma$-finite measure case
would meet essential difficulties. In addition, their results are
global from which one cannot easily derive the local version.

Our concern is to obtain a local gradient estimate for harmonic
functions on $RCD^*(K,N)$ mms with $\sigma$-finite measures. 
For this purpose, we shall first need a local version of  Bochner inequality for $RCD^*(K,N)$ mms.  The global version was proven by
Erbar-Kuwada-Sturm \cite{ErbarKuwadaSturm13}. Here ``global'' means the statement of their Bochner inequality is only valid
for global $W^{1,2}$-functions. 
 Note that Zhang-Zhu \cite{Zhang2012} proved a
similar Bochner inequality on Alexandrov spaces with Ricci curvature
bounded below. A delicate local structure, so-called
$DC$-differential structure (see e.g. \cite{Perelman}), of Alexandrov spaces
plays an essential role in the proof of \cite{Zhang2012}. This rules
out the possibility of this strategy in our setting. Nevertheless,
we can choose nice cut-off functions and apply the global Bochner
inequality proven by Erbar-Kuwada-Sturm \cite{ErbarKuwadaSturm13} to
derive the local one. This is the novelty of our
approach. An important ingredient we need is the local Lipschitz
regularity for $W_{\rm loc}^{1,2}$ functions with Laplacian in
$L^p$, which was obtained independently by Jiang \cite{Jiang2013}
and the second author \cite{Kell2013} using the method initiated in
\cite{Koskela2003,Jiang2011}, see Lemma \ref{l:Lipschitz regularity}
below.

\begin{thm}[Local Bochner inequality] \label{t:Bochner and W12}
Let $(X, d, m)$ be an $RCD^*(K,N)$ mms. Let $u$ be a function in
$\mathcal{D}_{L^4_\loc}(\Delta)$ with $\Delta u\in W_{\loc}^{1,2}\cap
L_{\loc}^{p}(X,d,m)$ for $p>N$. Then
\[
|\nabla u|_w^{2}\in W_{\loc}^{1,2}(X,d,m)
\]
 and the Bochner inequality holds in the weak sense of measures
\begin{equation}
\mathcal{L}_{|\nabla u|_w^{2}}\ge2\left(\frac{(\Delta u)^{2}}{N}dm+\langle\nabla u,\nabla(\Delta u)\rangle dm+K|\nabla u|_w^{2}dm\right),\label{e:Bochner sigma finite}
\end{equation}
that is,
for all $\varphi\in W^{1,2}(X,d,m)$ with compact support we have
\begin{eqnarray*}
\int\langle\nabla\varphi,\nabla|\nabla u|_w^{2}\rangle dm
& \ge & 2\Big({\displaystyle \int\varphi\frac{(\Delta u)^{2}}{N}dm+\int\varphi\langle\nabla u,\nabla(\Delta u)\rangle dm}\\
 &  & +K\int\varphi|\nabla u|_w^{2}dm\Big).
\end{eqnarray*}
\end{thm}


\

One of our main result is Cheng-Yau type local gradient estimate for
harmonic functions on $RCD^*(K,N)$-mms. A function $u$ is called harmonic (subharmonic resp.) on an open set
$\Omega\subset X$ if $u\in W_{\loc}^{1,2}(\Omega)$ and
\[
\int_{\Omega}\langle\nabla u,\nabla\varphi\rangle dm=0\ (\leq 0\
\mathrm{resp.})
\]
 for any $0\leq\varphi\in\Li(\Omega)$ with compact support. This is equivalent to say that $\mathcal{L}_u=0 \ (\geq 0\
\mathrm{resp.})$ (see Section 2).

\begin{thm}[Cheng-Yau type gradient estimate]\label{t:Yau's gradient estimate}
Let $(X, d, m)$ be an $RCD^*(K,N)$ mms for $K\leq 0$ and $N\in [1,\infty)$. Then there
exists a constant $C=C(N)$ such that every positive harmonic
function $u$ on geodesic ball $B_{2R}\subset X$ satisfies
\begin{eqnarray*}
\frac{|\nabla  u|_w}{u}\leq C\frac{1+\sqrt{-K}R}{R} \ \ \ \ \
\mathrm{in}\ B_R.
\end{eqnarray*}
\end{thm}

Since Bochner's technique using the maximum principle on Riemannian manifolds is not
available on metric spaces, we adopt the Moser iteration to prove
the local gradient estimate, following the idea of Zhang-Zhu
\cite{Zhang2012}. 
Note that for the case $K<0,$ one shall carry out a more delicate
Moser iteration as done by the first and the third authors in
\cite{HuaXia13, Xia2013}. In order to carry out the Moser iteration,
we need the regularity result, $|\nabla u|_w^2\in W^{1,2}_{\rm
loc}(X),$ for a harmonic function $u$. This follows from a general
result by Savar\'e using Dirichlet form calculation, see
\cite[Lemma~3.2]{Savare2013} or Lemma \ref{l:Bochner formula} below. 
For a different proof of this result
on Alexandrov spaces with Ricci lower bounds, we refer to
\cite[Theorem~1.2]{Zhang2012}. This will also be crucial in
order to prove Theorem~\ref{t:linear growth harmonics}, where we
essentially use the fact that $|\nabla u|_w^2\in W^{1,2}_{\rm loc}(X)$
is subharmonic for every harmonic function $u$ on
$RCD^{*}(0,N)$ mms.



Theorem \ref{t:Yau's gradient estimate} immediately yields Cheng's Liouville theorem for sublinear
growth harmonic functions on $RCD^*(0,N)$ spaces.  

\begin{coro}[Cheng's Liouville theorem]\label{c:sublinear growth}
On an $RCD^*(0,N)$ mms with $N\in [1,\infty)$, there are no nonconstant harmonic functions
of sublinear growth, i.e. if $u$ is harmonic and
$$\limsup_{R\to \infty}\frac{1}{R}\sup_{B_R}|u|=0$$
then it is constant.
\end{coro}


\

The second part of the paper is on dimension estimates for
spaces of polynomial growth harmonic functions on $RCD^*(K,N)$ mms.
The history leading to these results started in the study of Riemannian geometry.
Cheng-Yau's gradient estimate \cite{ChengYau75} implies that
sublinear growth harmonic functions on Riemannian manifolds with
nonnegative Ricci curvature are constant. Yau
further conjectured in \cite{Yau87,Yau93} that the space of polynomial
growth harmonic functions  on such
manifolds with growth rate less than or equal to $d$ should be of finite dimension. Colding-Minicozzi
\cite{ColdingMinicozzi97JDG,ColdingMinicozzi97,ColdingMinicozzi98Weyl}
gave an affirmative answer to Yau's conjecture in a very general
framework of weighted Riemannian manifolds utilizing volume doubling property and Poincar\'e
inequality which are even adaptable to general metric spaces. A simplified
argument by the mean value inequality can be found in
\cite{ColdingMinicozzi98,Li97} where the dimension estimates are
nearly optimal. This inspired many generalizations on
manifolds
\cite{Tam98,LiWang99,LiWang00,SungTamWang00,KimLee00,Lee04,ChenWang07}.
The crucial ingredients of these proofs are the volume
growth property and the Poincar\'e inequality (or mean value
inequality).

Let $H^q(X):=\{u\in W^{1,2}_{\rm loc}(X): \mathcal{L}_u=0,
|u(x)|\leq C(1+d(x,p))^q\}$ denote the space of polynomial growth
harmonic functions on $X$ with growth rate less than or equal to $q$
for some (hence all) $p\in X.$ Before stating the theorem, we shall
point out a main difference between harmonic functions on Riemannian
manifolds and those on other metric spaces. The unique continuation
property for harmonic functions on mms is unknown,
leaving us with the
 problem of verifying the inner product property of the following bilinear form
 $$\langle u,v\rangle_R=\int_{B_R}uv dm, \ \ u,v\in L^2(X,m),$$ where $B_R$ is a geodesic ball with radius $R.$
 We circumvent this difficulty by a lemma in \cite{Hua11}
 (see Lemma \ref{xxw2} below). By using the Bishop-Gromov volume comparison (see
Theorem~\ref{t:BishopGromov}) and the Poincar\'e inequality (see
Theorem~\ref{t:Poincare}) on $RCD^*(0,N)$ spaces, we obtain
the following optimal dimension estimate for 
$H^q(X)$.

\begin{thm}[Polynomial growth harmonic functions]\label{t:polynomial growth harmonics} Let $(X,d,m)$ be an $RCD^*(0,N)$ mms with $N\in [1,\infty)$. Then there exists some constant $C=C(N)$ such that
$$\dim H^q(X)\leq Cq^{N-1}.$$
\end{thm}

\

For the space of  linear growth harmonic functions, we can give more precise estimate.
\begin{thm}[Linear growth harmonic functions]\label{t:linear growth harmonics} Let $(X,d,m)$ be an $RCD^*(0,N)$ mms with $N\in [1,\infty)$ and $p\in X.$ Suppose the volume growth of $(X,d,m)$
satisfies
\begin{equation}\label{e:volume growth n}
\limsup_{R\to\infty}\frac{m(B_R(p))}{R^n}<\infty
\end{equation} for some $n\leq N,$ then
$$\dim H^1(X)\leq n+1.$$
\end{thm}

On Riemannian manifolds, Theorem \ref{t:linear growth harmonics} was studied by Li-Tam
\cite{LiTam89} and the equality case was characterized by
Cheeger-Colding-Minicozzi \cite{CheegerColdingMinicozzi95}. For the
investigation of linear growth harmonic functions, see also Wang
\cite{Wang95}, Li \cite{Li95} on K\"ahler manifolds and
Munteanu-Wang \cite{MunteanuWang11} on weighted Riemannian manifolds
with nonnegative Ricci curvature.

One of the key ingredients of the proof is the following so-called
mean value theorem at infinity for  nonnegative subharmonic functions (see Theorem
\ref{t:mean value theorem}),
$$\lim_{R\to \infty} \frac{1}{m(B_R)}\int_{B_R} udm=\ess \sup_X u.$$
 The original proof for this mean value theorem  by Li
\cite{Li86} (see also \cite[Lemma~16.4]{Li12}) used heat kernel
estimates. It can also be proved by a tricky monotonicity formula
involving the mean value of harmonic functions on geodesic spheres,
see \cite[Theorem 3.3]{MunteanuWang11} and \cite[Corollary
6.6]{Zhang2012}. However, these methods seem hard to be extended to general metric spaces. For our purpose, we present a new proof only using
the weak Harnack inequality for superharmonic functions, which is a
consequence of Moser iteration, see Theorem \ref{t:Moser iteration}
below. This will be the main ingredient to prove the optimal dimensional
bound of the space of linear growth harmonic functions.

We remark that since the proof of above theorem involves the Bochner inequality, Theorem \ref{t:linear growth harmonics}
is the first result on the dimension estimate of linear growth
harmonic functions on nonsmooth metric spaces, even on Alexandrov
spaces \cite{Hua09,Hua11,Jiao12}.

\

To summarize, we prove a local Bochner inequality on $RCD^*(K,N)$
mms. Then we adopt a delicate Moser iteration to show Cheng-Yau type
local gradient estimate of harmonic functions. By using the
Bishop-Gromov's volume comparison, the Poincar\'e inequality and the
Bochner inequality, we extend various optimal dimension estimates of
the spaces of polynomial growth harmonic functions on Riemannian
manifolds to a large class of nonsmooth mms satisfying the
$RCD^*(0,N)$ condition. To this extent, we provide a relatively
complete picture of global properties of harmonic functions on mms
with nonnegative Riemannian Ricci lower bound.

The paper is organized as follows: In Section~\ref{s:preliminaries}
we collect the basics of the analysis on mms with
Riemannian Ricci curvature bounds. Section~\ref{s:local Bochner} is
devoted to the proof of the local Bochner inequality, Theorem
\ref{t:Bochner and W12}. In Section~\ref{s:Cheng Yau gradient} we
 prove the Cheng-Yau type
gradient estimate on RCD spaces, Theorem \ref{t:Yau's gradient
estimate}. In the last section, we prove the optimal dimension
estimates of the spaces of polynomial growth harmonic functions and
linear growth harmonic functions, Theorem \ref{t:polynomial growth
harmonics} and Theorem \ref{t:linear growth harmonics} respectively.

\

\section{Preliminaries}\label{s:preliminaries}

We will only introduce some necessary notations and refer to \cite{Gigli2012,AGS2011a}
for proofs of the statements and further references.

Throughout the paper we assume $(X,d,m)$ is a metric measure spaces
and the measure $m$ is $\sigma$-finite and satisfying a maximum
growth bound, i.e. for some $C>0$
\[
|B_{r}(x)|\leq C\cdot e^{Cr^{2}},
\]
where $|B_r(x)|$ is an abbreviation for $m(B_r(x))$.
 Additionally assume that $(X,d)$ is a locally compact length space.
Both assumptions simplify the following statements. Since any $RCD$ mms will satisfy them they are in no way restrictive.

For the  subset
of  $L^{2}(X,m)$ containing all Lipschitz functions with compact support
one can define a (minimal) \textit{weak upper gradient} $|\nabla f|_{w}$ (see e.g.  \cite{AGS2011a}). The
\textit{Cheeger energy} $\Ch$ is defined by
\[
\Ch(f)=\int|\nabla f|_{w}^{2}dm.
\]
The subset of $L^{2}$-functions with finite Cheeger energy will be
denoted by $W^{1,2}(X,m)$. Equipped with the norm
$$\|u\|^2_{W^{1,2}}:=\|u\|^2_{L^2}+{\rm Ch}(u),$$ $W^{1,2}(X,m)$ is a Banach space. It can be shown that all Lipschitz
functions with compact support have weak upper gradients in
$L^{\infty}(X,m)$ and are contained in $W^{1,2}(X,m)$. Because the
weak upper gradient is a local object, there is also a well-defined
notation of $W_{\loc}^{1,2}(X,m)$: For any open set $\Omega\subset
X$, $W^{1,2}(\Omega)$ is the set of functions whose weak upper
gradient restricting to $\Omega$ have finite $L^2(\Omega)$ norm; The
set $W_{\loc}^{1,2}(\Omega)$ is defined as the set of functions
which belongs to $W^{1,2}(\Omega')$ for any precompact open set
$\Omega'\subset \Omega$. We denote by $\Li_c(\Omega)$ and
$W^{1,2}_c(\Omega)$ the set of functions in $\Li(\Omega)$ and
$W^{1,2}(\Omega)$ with compact support in $\Omega$ respectively.

It can be shown that the Cheeger energy $\Ch$ is convex and lower
semicontinuous. Thus the natural gradient flow $P_{t}:L^{2}\to
L^{2}$, called heat flow, can be defined. By the calculus developed
in \cite{AGS2011a} one can define a natural Laplace operator
$\Delta$ on a dense subset of $L^{2}$ as subdifferential of $\Ch$.
Furthermore, the following holds
\[
P_{t}\;\mbox{is linear}~\Longleftrightarrow~\Ch~\mbox{is a quadratic
form}~\Longleftrightarrow ~ \Delta \mbox{ is linear},
\]
A mms whose Cheeger energy is quadratic will be called \textit{infinitesimal
Hilbertian}.

Our main focus will be the following subset of infinitesimal Hilbertian
spaces. For this we also assume throughout that every function $f\in W^{1,2}(X,m)$ with
$|\nabla f|_w \le 1$ has a representative which is $1$-Lipschitz. Whenever this
holds for the space $(X,d,m)$ we say the space satisfies the 
 the Sobolev-to-Lipschitz property (see \cite{Ambrosio2012,Gigli2013a}). 
Indeed, without this condition the generalized gradient $|\nabla f|$ might 
be $0$ for all $L^2$-functions and hence the condition below trivially satisfied.

\begin{defn}[$RCD^*(K,N)$ mms]\label{RCD}
 We say an infinitesimal Hilbertian metric measure
spaces satisfying the Sobolev-to-Lipschitz property is a (finite-dimensional) 
$RCD^*(K,N)$ mms or satisfies the
$RCD^*(K,N)$ condition for some $K\in\mathbb{R}$ and $N>0$, if for
any $f\in W^{1,2}(X,m)$ and  we have $m$-a.e. in $X$,
\[
|\nabla P_{t}f|_{w}^{2}+\frac{4Kt^{2}}{N(e^{2Kt}-1)}|\Delta
P_{t}f|^{2}\le e^{-2Kt}P_{t}(|\nabla f|_{w}^{2}).
\]
\end{defn}

\begin{rem*}
By \cite[Theorem 7]{ErbarKuwadaSturm13} this is equivalent to the
more classical $CD^*(K,N)$ condition defined via Wasserstein geodesics
or the Bochner inequality (\ref{e:Bochner formula}) defined below.
Furthermore, any such space is a proper geodesic space, i.e. all
bounded set are precompact and between each two points there is a
rectifiable curve whose length is the distance of those points.
\end{rem*}

Because $\Ch$ is a quadratic form in $W^{1,2}(X,m)$,  there is an associated Dirichlet form $\mathcal{E}$, i.e.,
$\mathcal{E}: W^{1,2}(X,m)\times W^{1,2}(X,m)\to \mathbb{R}$ is the unique bilinear symmetric form satisfying $$\mathcal{E}(f,f)=\hbox{Ch}(f).$$

It was proven (see Section 4.3 in \cite{AGS2011}) that the Dirichlet
form can be written as
\[
\mathcal{E}(u,v)=\frac{1}{2}\int\langle\nabla u,\nabla v\rangle dm,\quad u,v\in W^{1,2}(X,m)
\]
where
\[
\langle\nabla u,\nabla v\rangle(x)=\lim_{\epsilon\to0}\frac{|\nabla(v+\epsilon u)|_{w}^{2}-|\nabla v|_{w}^{2}}{2\epsilon}.
\]
The notion $\langle\nabla u,\nabla v\rangle$ should be understood as a
bilinear and symmetric map from $W^{1,2}\times W^{1,2}\to L^{1}$.
We remark that $|\nabla u|_w^{2}=\langle\nabla u,\nabla u\rangle$ and $\langle\nabla u,\nabla v\rangle\in L^{1}$ are a well-defined objects whereas
$\nabla u$ is not. Furthermore, it is not difficult to show that under the assumptions above the Dirichlet form $\mathcal{E}$ is strongly local, closed and Markovian.

Gigli \cite[4.4,4.7]{Gigli2012} also showed that there is a dense subclass
denoted by $D(\mathcal{L})\subset W_{\loc}^{1,2}(X,m)$ that admit a measure-valued Laplacian, i.e. for every $u\in D(\mathcal{L})$ there is a Radon
measure denoted by $\mathcal{L}_u$ such that for all
$v\in \Li_c(X)$
\[
\mathcal{E}(v,u)=-\frac{1}{2}\int vd\mathcal{L}_{u}.
\]
If $\mathcal{L}_{u}$ has local $L^{p}$-density
w.r.t. $m$ we denote its density by $\Delta u$ and $\Delta u$ will be
 called the Laplacian of $u$.
The subset of functions admitting $L^p_\loc$-Laplacians will be
denoted by $\mathcal{D}_{L^p_\loc}(\Delta)$ (resp. $\mathcal{D}_{L^p}(\Delta)$ if $u\in W^{1,2}(X,m)$ and $\Delta u\in L^{p}(X,m)$). Note that the $L^2$-Laplacian agrees with the generator of the Dirichlet form $\mathcal{E}$.
\begin{defn}[Harmonic, subharmonic and superharmonic functions]\label{d:harmonic}
A function $u$ is called harmonic (subharmonic, superharmonic resp.)
on the domain $\Omega$ if $u\in W_{\loc}^{1,2}(\Omega)$ and
\[
-\int_{\Omega}\langle\nabla u,\nabla\varphi\rangle dm=0\ (\geq0,\leq0\ \mathrm{resp.})
\]
 for any $0\leq\varphi\in\Li_c(\Omega)$. 
 \end{defn}
\begin{rem*}
It turns out (sub-/super-)harmonicity implies $u\in D(\mathcal{L})$ (see \cite[Prop. 4.13]{Gigli2012}) so that it is equivalent to
\[
\mathcal{L}_{u}=0\ (\geq0,\leq0\ \mathrm{resp.}).
\]

\end{rem*}
By the definition of harmonic (subharmonic) functions and
integration by parts with cut-off functions, we obtain the following
Caccioppoli inequality. We omit the proof here.
\begin{lem}
[Caccioppoli inequality]\label{l:Caccio} Let $(X,d,m)$ satisfy
$RCD^*(K,N)$ condition. Then for any nonnegative subharmonic
function $u$ on $B_{2R},$ we have
\[
\int_{B_{R}}|\nabla u|_w^{2}dm\leq
\frac{C}{R^{2}}\int_{B_{2R}}u^{2}dm,
\]
 where $C=C(N).$
\end{lem}
We summarize the local calculus of the weak upper gradient and the
Laplace operator as follows.
\begin{thm}[\cite{Gigli2012,AGS2011a}]\label{t:basic properties of weak upper gradient} Assume $(X,d,m)$ is an infinitesimal Hilbertian mms. Assume $u,v,w\in W_{\loc}^{1,2}(X,m)$ Then the following holds:
\begin{enumerate}
\item $|\nabla u|_{w}=|\nabla\tilde{u}|_{w}\quad m\mbox{-a.e. on }\{u=\tilde{u}\}$
and $|\nabla u|_{w}=0$ $m$\textup{-a.e. on $\{u=c\}$ }for
$c\in\mathbb{R}.$
\item $u\in W_{\loc}^{1,2}(X,m)$ iff $u\cdot\chi\in W^{1,2}(X,m)$ for all
$\chi\in\operatorname{Lip}_{c}$. Moreover, if $u$ has compact
support then it is in $W^{1,2}(X,m).$
\item Assume $u\in\mathcal{D}_{\loc}(\Delta).$ The Laplace operator is a
local object, i.e. if $\Omega\subset X$ is open and
$\{\Omega_{i}\}_{i\in I}$ an open covering of $\Omega$ then

\begin{enumerate}
\item $\mathcal{L}_{u|\Omega}=\mu$ iff for all $v\in \Li_c(\Omega)$
\[
\mathcal{E}(v,u)=-\frac{1}{2}\int vd\mu,
\]

\item set $\mathcal{L}_{u|\Omega_{i}}=\mu_{i}$; if $\mu_{i|\Omega_{i}\cap\Omega_{j}}=\mu_{j|\Omega_{i}\cap\Omega_{j}}$
whenever $\Omega_{i}\cap\Omega_{j}\neq\varnothing$ then
\[
(\mathcal{\mathcal{L}}_{u|\Omega})_{\Omega\cap\Omega_{i}}=\mu_{i}.
\]

\end{enumerate}
\item $|\nabla\cdot|_{w}$ and $\Delta$ satisfy the chain rule, i.e. if
$\phi:\mathbb{R}\to\mathbb{R}$ is Lipschitz then
\[
|\nabla\phi(u)|_{w}=|\phi'(u)||\nabla u|_{w}\quad m\mbox{-a.e.}
\]
and if, in addition, $u\in \mathcal{D}_{L^p_\loc}(\Delta)$ then
\[
\Delta\phi(u)=\phi'(u)\Delta u+\phi''(u)|\nabla u|_{w}^{2}\quad m\mbox{-a.e.}
\]

\item The inner product $\langle\nabla\cdot,\nabla\cdot\rangle$ and $\Delta$
satisfy the Leibniz rule, i.e. if $u,v\in W_{\loc}^{1,2}\cap
L_{\loc}^{\infty}(X,m)$ then
\[
\langle\nabla(u\cdot v),\nabla w\rangle=v\langle\nabla u,\nabla w\rangle+u\langle\nabla v,\nabla w\rangle\quad m\mbox{-a.e.}
\]
and if, in addition $u,v \in \mathcal{D}_{L^1}(\Delta)$, then
\[
\Delta u\cdot v = v\cdot\Delta u+u\cdot\Delta v+2\langle\nabla u,\nabla v\rangle
\]
so that $u\cdot v \in \mathcal{D}_{L^1}(\Delta)$.

\item The Cauchy-Schwarz inequality holds, $$|\langle\nabla u,\nabla v\rangle|\leq |\nabla u|_w|\nabla v|_w.$$
\end{enumerate}

\end{thm}

\begin{proof}
This follows directly from \cite{Gigli2012}. More precisely, for $(1)$ see \cite[Equation (2.15)]{Gigli2012}, for $(2)$ see \cite[2.6-(i)]{Gigli2012}, for $(3)$ see \cite[3.15]{Gigli2012} and for $W^{1,2}$-functions $(4)$ and $(5)$ are proven in \cite[3.15,3.17]{Gigli2012}. The general cases then hold by approximation and using test functions with compact support.
Finally fact $(6)$ is a consequence of sublinearity of $f\mapsto |\nabla f|$.
Alternatively the results follow from Dirichlet theory (see \cite{Fukushima1994}).
\end{proof}
Directly from the chain rule we obtain.
\begin{cor}
Assume $u\in W^{1,2}(X,m)$ is continuous and harmonic. Then for any convex $C^2$-function $\phi:\mathbb{R} \to \mathbb{R}$ 
the funtion $\phi(u)$ is subharmonic.
\end{cor}



From calculus rules above one easily sees that $W^{1,2}(X,m)\cap L^{\infty}(X,m)$ and its local version are algebras, i.e. if $u,v\in W^{1,2}(X,m)\cap L^{\infty}(X,m)$ then $u\cdot v\in W^{1,2}(X,m)\cap
L^{\infty}(X,m)$. Similarly, $\mathcal{D}_{L^p_\loc}(\Delta)\cap \Lip_\loc(X,d)$  is an algebra.

The following was proved in \cite{AMS14}. It will be a central part in showing that cut-off function behave nicely which will then be use
to prove Theorem \ref{t:Bochner and W12}.
\begin{lem}[{\cite[Theorem 5.5]{AMS14}}] \label{lem:AMS}
Assume $(X,d,m)$ is an $RCD^*(K,N)$-space. Then $u\in D_{L^4}(\Delta)\cap L^4(X,m)$ implies
$|\nabla u|^2_w \in W^{1,2}(X,m)$.
\end{lem}

Now we construct the cut-off functions.

\begin{prop}
Let $(X,d,m)$ be an $RCD^*(K,N)$ space and $K$ any compact subset in $X$.
There is a compactly supported Lipschitz function
$\Psi=\Psi_K:X\to\mathbb{R}$ with
$\Psi \in \mathcal{D}_{L^\infty}(\Delta)$
such that $\Delta\Psi\in W^{1,2}_c(X,d,m)$
and $\Psi$ equals $1$ in a neighborhood of $K$.
\end{prop}
\begin{proof}
This follows essentially from \cite[Lemma 6.7]{AMS14}. We give a short
argument: Let $\phi$ be any compactly supported Lipschitz which equals
$1$ in a neighborhood of $K$.
Let $\phi_\epsilon = P_\epsilon \phi$. By the gradient estimate the functions
${\phi_\epsilon}$,${\epsilon\in(0,1]}$, are Lipschitz functions with constant bounded by
$\max\{e^{-K},1\}\Lip(\phi)$.

By the smoothing properties of the heat flow we have
$\phi_\epsilon \in \mathcal{D}_{L^2}(\Delta)$. Linearity of the heat flow
implies $\Delta(\phi_\epsilon)$ is in the image of
$P_{\frac{\epsilon}{2}}$ so that $\Delta \phi_\epsilon \in \Lip(X,d)\cap L^\infty(X,m)$ and hence $\phi_\epsilon \in D_{L^\infty}(\Delta)$ with $\Delta\phi \in W^{1,2}(X,m)$.

As $\phi_\epsilon \to \phi$ in $L^2$. The Arzel\`a-Ascoli theorem implies
that this convergence is actually locally uniformly. Using the
Gaussian estimate on the heat kernel one can show that the convergence is
globally uniformly. In particular, there are open neighborhoods
$U_1\subset U_2$ of $K$ such that $\phi_\epsilon \ge 1-\delta$ on $U_1$
and $\phi_\epsilon \le \delta$ outside of $U_2$ for any
$\epsilon<\epsilon_0$.

Now let $\eta:\mathbf{R}\to \mathbf{R}$ be a smooth cut-off function
which equals $1$ for $r\ge 1-\delta$ and $0$ for $r \le \delta$.
We claim that $\eta(\phi_\epsilon)$ for some $\epsilon < \epsilon_0$
is the required cut-off functions. Indeed, by
assumption it is equal to $1$ in a neighborhood of $K$ and has compact support.
By chain rule we have
\[
\Delta \eta(\phi_\epsilon) = \eta'(\phi_\epsilon)\Delta \phi_\epsilon+\eta''(\phi_\epsilon)|\nabla \phi_\epsilon|_{w}^{2}.
\]

Note $|\nabla \phi_\epsilon|_{w}^{2}$ that is bounded and by 
Lemma \ref{lem:AMS} in $W^{1,2}(X,m)$. Therefore,
$$\Delta \eta(\phi_\epsilon) \in L^\infty(X,m)\cap W^{1,2}(X,m).$$
\end{proof}

Denote by $\cutoff$ the set of Lipschitz functions with bounded support and $W^{1,2}\cap L^\infty$-Laplacian, i.e. $\varphi \in \cutoff$ if $\varphi \in D_{L^\infty}(\Delta)\cap \Lip_c$ with $\Delta \varphi \in W^{1,2}(X,m)$.

\begin{pro}\label{proX}
Assume $(X,d,m)$ is an $RCD^*(K,N)$-space and
$u \in \mathcal{D}_{L^4_\loc}(\Delta)\cap \Lip_\loc(X,d)$
then the following is true for every $\psi \in \cutoff$
\begin{enumerate}
\item $\psi u \in \mathcal{D}_{L^4}(\Delta)$. In particular, $|\nabla (\psi u)|_w^2 \in W^{1,2}(X,m)$
\item $\langle\nabla\psi,\nabla u\rangle\in W^{1,2}(X,m)$
\end{enumerate}
\end{pro}

Before proving this proposition, we state the following corollary which follows from the last point together with Leibniz rule.
\begin{cor} \label{cor:cutoff-regularity}
Assume $(X,d,m)$ is an $RCD^*(K,N)$-space and
$u\in \mathcal{D}_{L^4_\loc}(\Delta)\cap \Lip_\loc(X,d)$ with
$\Delta u\in W^{1,2}_\loc(X,m)$
then $\Delta(\psi u)\in W^{1,2}(X,m)$ for all $\psi \in \cutoff$.
\end{cor}

\begin{proof}[Proof of Proposition \ref{proX}]
The first item follows from Leibniz rule together with the previous lemma. Indeed, Leibniz rule implies
\[
\Delta(\psi u)= \psi \Delta u+ u \Delta \psi + 2 \langle \nabla \psi,\nabla u \rangle.
\]
As $u,\psi \in \Lip_\loc \cap \mathcal{D}_{L^4_\loc}(\Delta)$ and $\psi u$ has bounded support we see that $\Delta(\psi u) \in L^4(X,d)$. By Lemma
\ref{lem:AMS} above this implies $|\nabla (\psi u)|_w^2\in W^{1,2}(X,d)$.

For the second fact choose $\Phi\in\cutoff$ which equals $1$ in a neighborhood $U$ of the support of $\psi$. Because $\Phi \equiv 1$ in $U$,
 it holds
\[
4\langle \nabla \psi, \nabla u\rangle = |\nabla \Phi(u+\psi)|_w^2 -|\nabla \Phi(u-\psi)|_w^2 \mbox{ in } U
\]
with support in the interior of $U$.

From the first fact $|\nabla \Phi(u \pm \psi)|_w^2 \in W^{1,2}(X,m)$ so that $v=|\nabla \Phi(u+\psi)|_w^2 -|\nabla \Phi(u-\psi)|_w^2 \in W^{1,2}(\bar U, m)$ with support strictly in the interior of $U$. But then $\tilde v \in W^{1,2}(X,m)$ where $\tilde v$ is the trivial extension of $v$ outside of $U$. But $\tilde v = 4\langle \nabla \psi, \nabla u\rangle$ almost everywhere.
\end{proof}

Next, we recall some basic geometric properties of $RCD^*(K,N)$
mms. The following volume comparison theorem is well-known and follows
from the $RCD^*(K,N)$ (\cite[Theorem 2.3]{Sturm06b}, compare also \cite[Remark 3.5]{ErbarKuwadaSturm13}).

\begin{thm}[Bishop-Gromov volume comparison]\label{t:BishopGromov}
Let $(X,d,m)$ be an $RCD^{*}(K,N)$ space with $K\leq 0$ and $N\in [1,\infty)$. Then for any
$0<r<R<\infty,$
\begin{equation}\label{e:BishopGromov}
\frac{|B_R(x)|}{|B_r(x)|}\leq \left\{
\begin{array}{lll}\frac{\int_0^R \sinh^{N-1} \left(\sqrt{\frac{-K}{N-1}}t\right)}{\int_0^r \sinh^{N-1} \left(\sqrt{\frac{-K}{N-1}}t\right)},& K<0, 1<N<\infty,

\\ \left(\frac{R}{r}\right)^N,& K=0
\\ \frac{R}{r}, &K\leq 0, N=1. 
\end{array}
\right.
\end{equation}
\end{thm}

In order to prove Cheng-Yau's local gradient estimate for the space satisfying
$RCD^*(K,N)$ condition with $K<0,$ we need the Poincar\'e and
Sobolev inequalities with precise dependence on $K$ and radii of the
balls.

The Poincar\'e inequality was proved by several authors, see
Lott-Villani \cite{LottVillani07}, von Renesse \cite{Renesse08},
Rajala  \cite[Theorem~1.1]{RajalaJFA12}, \cite{RajalaCV12} for mms satisfying $RCD^*(K,N)$. A priori, those results need non-branching
geodesics. However, the same proofs also work for so-called
essentially non-branching spaces which is known to hold under
the $RCD^*(K,N)$-condition \cite{AmbrosioGigliMondinoRajala12}.

\begin{thm}[Poincar\'e inequalty]\label{t:Poincare}
Let $(X,d,m)$ be an $RCD^*(K,N)$ mms for $K\leq0$ and $N\in [1,\infty)$. For any
$1\leq p<\infty,$ there exists a constant $C=C(N,p)$ such that for
all $u\in W^{1,2}(B_{2R}),$
\begin{equation}\label{e:Poincare CD(K,N)}
\int_{B_R}|u-u_{B_R}|^pdm \leq CR^pe^{C\sqrt{-K}
R}\int_{B_{2R}}|\nabla u|_w^p dm,
\end{equation} where $u_{B_R}=\frac{1}{|B_R|}\int_{B_R}u dm.$
\end{thm}

\begin{proof}
The references only prove the result for upper gradients.
However, one obtains the above via the following argument:
The local Lipschitz is an upper gradient so the $(p,p)$-Poincar\'e inequality holds under the $RCD^*(K,N)$-condition.
This implies that for Lipschitz functions the local Lipschitz constant is (almost everywhere) equal to the $p$-weak upper gradient. By approximation we see that the Poincar\'e inequality holds also for the $p$-weak upper gradient.
In addition, one can show that the $p$-weak upper gradient for any $p>1$, is independent of $p$ and is therefore equal to the $2$-weak upper gradient which we denoted by $|\nabla \cdot |_w$. See also \cite{Gigli2014a} for a more direct proof of this fact in case the weaker condition $RCD^*(K,\infty)$ holds.
The case $p=1$ follows by a limiting argument.
\end{proof}


\

\section{A local Bochner inequality}\label{s:local Bochner}

In the rest of the paper, we assume throughout that $(X,d,m)$ is an
$RCD^*(K,N)$ mms for $K\leq 0$ and $N\in [1,\infty)$

The following lemma was proven in \cite{ErbarKuwadaSturm13}. We will
include a more technical statement of Savar\'e \cite{Savare2013}.

\begin{lem}[Bochner formula, {\cite[Theorem 5]{ErbarKuwadaSturm13}}, {\cite[Lemma 3.2]{Savare2013}}]\label{l:Bochner formula}Let $(X,d,m)$ be an $RCD^*(K,N)$ mms. For any $u\in D(\Delta)$ with $\Delta u\in
W^{1,2}(X,d,m)$ and all bounded and nonnegative $g\in D(\Delta)$
with $\Delta g\in L^{\infty}(X,m)$ we have
\begin{equation}
\frac{1}{2}\int\Delta g|\nabla u|_w^{2}dm\geq\frac{1}{N}\int g(\Delta u)^{2}dm+\int g\langle\nabla u,\nabla(\Delta u)\rangle dm+K\int g|\nabla u|_w^{2}dm.\label{e:Bochner formula}
\end{equation}

Furthermore, if $u\in\operatorname{Lip}(X)\cap L^{\infty}(X)$ then
$|\nabla u|_w^{2}\in W^{1,2}(X,d,m)\cap L^{\infty}(X,m)$ and
\[
\mathcal{L}_{|\nabla u|_w^{2}}\ge2\left(\frac{(\Delta u)^{2}}{N}dm+\langle\nabla u,\nabla(\Delta u)\rangle dm+K|\nabla u|_w^{2}dm\right).
\]

\end{lem}

Before proving the local version of Bochner inequality, Theorem \ref{t:Bochner and W12}, we point out that every function
with Laplacian in $L^{p}(X,m)$ for $p>N$ is locally Lipschitz continuous.
This was first proven for Ahlfors regular spaces in \cite{Koskela2003,Jiang2011}
but also holds in our setting (see \cite{Jiang2013} for the $L^{\infty}$
case and more general spaces and \cite{Kell2013} for necessary adjustments
for the case $p>N$ in our setting).
\begin{lem}
[Lipschitz regularity
\cite{Koskela2003,Jiang2011},\cite{Jiang2013,Kell2013}]\label{l:Lipschitz
regularity} Let $(X,d,m)$ be an $RCD^*(K,N)$ mms. Any $u \in
W_{\loc}^{1,2}(X,d,m)$ with $\Delta u\in L_{\loc}^{p}(X,d,m)$ and
$p>N$ is locally Lipschitz continuous.
\end{lem}
\begin{rem}
Note that \cite{Jiang2013,Kell2013} assume global $N$-doubling which only holds in our setting if $\inf_{x\in X} m(B_{R_0}(x))>0$. However, if one applies a version of the heat kernel comparison as in \cite[Equation (2.2)]{Jiang2013}, $N$-doubling only needs to hold local uniformly.
\end{rem}

Now we are ready to prove Theorem \ref{t:Bochner and W12}.

\begin{proof}[Proof of Theorem \ref{t:Bochner and W12}]
It suffices to show that the Bochner inequality holds locally, that
is it holds for all $\varphi\in W^{1,2}_c(X,d,m)$.

Let $U=U_\varphi$ be a neighborhood of $\supp \varphi$ and $\Psi\in\cutoff$ such that $\Psi \equiv 1$ in $U$.
By Lemma \ref{l:Lipschitz regularity} we know that $u$ is locally Lipschitz continuous.  Thus we can apply Corollary \ref{cor:cutoff-regularity} to
get $\tilde{u} := \Psi \cdot u \in \mathcal{D}_{L^p}(\Delta)$ with
$\Delta\tilde{u} \in W^{1,2}(X,m)$. Hence the Bochner
inequality holds for $\tilde u$. In particular,

\begin{eqnarray*}
\int\langle\nabla\varphi,\nabla|\nabla u|_w^{2}\rangle dm & = & {\displaystyle \int\varphi d\mathcal{L}_{|\nabla u|_w^{2}}}\\
 & \ge & 2\Big({\displaystyle \int\varphi\frac{(\Delta u)^{2}}{N}dm+\int\varphi\langle\nabla u,\nabla(\Delta u)\rangle dm}\\
 &  & +K\int\varphi|\nabla u|_w^{2}dm\Big).
\end{eqnarray*}
where $\chi$ is the indicator function of $U$. Since this holds for
any test function with support in $U$ we also have
\[
(\mathcal{L}_{|\nabla \tilde{u}|_w^{2}})_U \ge2 \chi_U \left(\frac{(\Delta \tilde{u})^{2}}{N}dm+\langle\nabla \tilde{u},\nabla(\Delta \tilde{u})\rangle dm+K|\nabla \tilde{u}|_w^{2}dm\right)
\]
We conclude by noting that $\tilde{u} = u$ in $U$.
\end{proof}

\

\section{Cheng-Yau type local gradient estimate}\label{s:Cheng Yau gradient}
In this section, we shall
prove Cheng-Yau type local gradient estimate for harmonic functions on $RCD^*(K,N)$ mms. We
denote by $B_R=B_R(x_0)$ the open ball in $X$ centered at some
point $x_0\in X$ with radius $R$.

For $RCD^{*}(0,N)$ mms, the standard Moser iteration using the
volume doubling property (implied by the Bishop-Gromov volume
comparison \eqref{e:BishopGromov}) and the Poincar\'e inequality
\eqref{e:Poincare CD(K,N)} yield the following Harnack inequality,
see e.g. \cite{HanLin97,Hua09,Hua11}. Note that the proofs in those papers
only require the doubling and adapt to the Sobolev calculus presented in this paper.

\begin{thm}\label{t:Moser iteration}
Let $(X,d,m)$ be an $RCD^*(0,N)$ mms. Then
\begin{enumerate}[(a)]
\item (mean value inequality) there exists a constant $C=C(N)$ such that for any subharmonic function $u$ on $B_{2R}$
\begin{equation}\label{mvi}
\ess \sup_{B_R} u\leq C\fint_{B_{2R}}|u|dm,
\end{equation} where $\fint_{B_{2R}}|u|dm=\frac{1}{|B_{2R}|}\int_{B_{2R}}|u|dm.$

\item (weak Harnack inequality) there exists a constant $C=C(N)$ such that for any nonnegative superharmonic function $u$ on $B_{2R}$
\begin{equation*}
\ess \inf_{B_R} u\geq C\fint_{B_{2R}}udm.
\end{equation*}

\item (Harnack inequality) there exists a constant $C=C(N)$ such that for any nonnegative harmonic function $u$ on $B_{2R}$
\begin{equation*}
\sup_{B_R} u\leq C \inf_{B_R} u.
\end{equation*}
\end{enumerate}
\end{thm}

In order to prove Theorem \ref{t:Yau's gradient estimate},
 we need to adopt a delicate Moser
iteration based on the local Bochner inequality. By the
Bishop-Gromov volume comparison and the Poincar\'e inequality for
$RCD^*(K,N)$ mms above, we can prove the following Sobolev
inequality, see e.g. \cite[Lemma 3.2]{MunteanuWang11}.
\begin{thm}[local uniform Sobolev inequality]\label{Sob}Let $(X,d,m)$ satisfy $RCD^*(K,N),$ with $K\leq 0$. Then there exist two constants $\nu>2$ and $C$, both depending only on $N$, such that for $B_{R}\subset X$ and $u\in W_{\loc}^{1,2}(B_R)$,
\begin{eqnarray}\label{Sobolev0}
\left(\int_{B_R} (u-u_{B_R})^{\frac{2\nu}{\nu-2}}dm
\right)^{\frac{\nu-2}{\nu}} \leq
e^{C(1+\sqrt{-K}R)}R^2|B_R|^{-\frac2\nu}\int_{B_R}|\nabla
u|_w^2dm,\end{eqnarray} where $u_{B_R}=\fint_{B_R}udm.$ In particular,
\begin{eqnarray}\label{Sobolev}
\left(\int_{B_R} u^{\frac{2\nu}{\nu-2}}dm
\right)^{\frac{\nu-2}{\nu}} \leq
e^{C(1+\sqrt{-K}R)}R^2|B_R|^{-\frac2\nu}\int_{B_R} (|\nabla
u|_w^2+R^{-2}u^2)dm.\end{eqnarray}
\end{thm}

By the Sobolev inequality above, we may adopt a delicate Moser
iteration as in Hua-Xia \cite{HuaXia13} to prove Cheng-Yau's local gradient
estimate, Theorem \ref{t:Yau's gradient estimate}. For the
completeness, we include the proof here.

\begin{proof}[Proof of Theorem \ref{t:Yau's gradient estimate}]
By Lemma \ref{l:Lipschitz regularity} $u$ is locally Lipschitz continuous.
Thus $u_\epsilon=u+\epsilon \in\Li(B_{2(R-\epsilon)}$ and 
$u_\epsilon\ge\epsilon$ on $B_{2(R-\epsilon)}$. 
If for $u_\epsilon$ the claim holds then
\[
\frac{|\nabla  u|_w}{u+\epsilon}\leq C(N)\frac{1+\sqrt{-K}(R-\epsilon)}{R-\epsilon} \ \ \ \ \
\mathrm{in}\ B_{R-\epsilon}.
\]
then letting $\epsilon\to 0$ implies the claim for $u$.

So without loss of generality assume, in addition, that  
$u\in\Li(B_{2R})$ and $u\ge\epsilon$ on $B_{2R}$. 
Theorem \ref{t:Bochner and W12} yields that $|\nabla u|_w^2\in W_{\rm loc}^{1,2}(B_{2R})$. Set $v:=\log u$. One can easily verify that 
\begin{eqnarray}\label{Peq1}
\mathcal{L}_v=-|\nabla v|_w^2\cdot dm.
\end{eqnarray}

Since $u\ge\epsilon$, $v\in \Li(B_{2R})$ and by setting $f=|\nabla v|_w^2$, it follows
from the Bochner inequality \eqref{e:Bochner sigma finite} in
Theorem~\ref{t:Bochner and W12} that for any $0\leq\eta\in
\Li_c(B_{2R})$,
\begin{eqnarray}\label{Peq2} \int_{B_{2R}} \langle\nabla\eta, \nabla f\rangle dm  \leq \int_{B_{2R}} \eta \left(2\langle \nabla v, \nabla f\rangle-2Kf-\frac{2}{N}f^2\right)dm.\end{eqnarray}

In fact, by an approximation argument, \eqref{Peq2} holds for any
$0\leq\eta\in W^{1,2}_c(B_{2R})\cap L^{\infty}(B_{2R}).$ Let
$\eta=\phi^2f^\beta$, with $\phi\in \Li_c(B_{2R})$, $0\leq\phi\leq
1$ and $\beta\geq
1$. 
Then $\eta$ is an admissible test function for \eqref{Peq2}. Hence
we have from \eqref{Peq2} that
\begin{eqnarray*}\label{Peq3}
&&\int_{B_{2R}} \left(\beta\phi^2f^{\beta-1}|\nabla f|_w^2+2\phi
f^\beta \langle\nabla f,\nabla \phi\rangle\right) dm\nonumber\\&\leq
&\int_{B_{2R}} \phi^2f^\beta \left(2\langle \nabla v, \nabla
f\rangle-2Kf-\frac{2}{N}f^2\right)dm .
\end{eqnarray*}

It follows from the Cauchy-Schwarz inequality ((7) in Theorem \ref{t:basic properties of weak upper gradient})  that
\begin{eqnarray*}\label{Peq7}
\frac{4\beta}{(\beta+1)^2}\int_{B_{2R}} \phi^2 |\nabla
f^{\frac{\beta+1}{2}}|_w^2dm  &\leq&\frac{4}{\beta+1}\int_{B_{2R}}
\phi f^{\frac{\beta+1}{2}}|\nabla\phi|_w  |\nabla
f^{\frac{\beta+1}{2}}|_wdm\nonumber\\&&+\frac{4}{\beta+1}\int_{B_{2R}}
\phi^2f^{\frac{\beta+2}{2}}|\nabla
f^{\frac{\beta+1}{2}}|_wdm\nonumber\\&&-\int_{B_{2R}}
\frac{2}{N}\phi^2f^{\beta+2}dm-\int_{B_{2R}} 2K\phi^2f^{\beta+1}dm.
\end{eqnarray*}
Using the H\"older inequality, we obtain
\begin{eqnarray*}\label{Peq8a} 
\int_{B_{2R}} \phi^2 |\nabla f^{\frac{\beta+1}{2}}|_w^2dm &\leq
&C\int_{B_{2R}} |\nabla\phi|_w^2  f^{\beta+1}dm+C\int_{B_{2R}}
\phi^2f^{\beta+2}dm\nonumber\\&&-C \beta\int_{B_{2R}}
\phi^2f^{\beta+2}dm-C\beta K\int_{B_{2R}} \phi^2f^{\beta+1}dm.
\end{eqnarray*}
We remark that from now on,  $C$ denotes various constants depend only
on $N$.

For $\beta$ sufficiently large, we can absorb the second term on the right hand side and get
\begin{eqnarray}\label{Peq8}
&&\int_{B_{2R}}  |\nabla (\phi
f^{\frac{\beta+1}{2}})|_w^2dm+C\beta\int_{B_{2R}} \phi^2f^{\beta+2}dm
\nonumber\\&\leq &2C\int_{B_{2R}} |\nabla\phi|_w^2
f^{\beta+1}dm-C\beta K\int_{B_{2R}}
\phi^2f^{\beta+1}dm.\end{eqnarray}
Using the Sobolev inequality \eqref{Sobolev}, we obtain
\begin{eqnarray}\label{Peq15}
&&\left(\int_{B_{2R}}\phi^{2\chi} f^{(\beta+1)\chi}
dm\right)^{\frac{1}{\chi}}  \leq
e^{C(1+\sqrt{-K}R)}R^2|B_{2R}|^{-\frac2\nu}\bigg(C\int_{B_{2R}}
|\nabla \phi|_w^2 f^{\beta+1}dm\nonumber\\&&\ \ \ \ \ \ \ \ \ \ \ \ \
\ \  \ \ \ +(CR^{-2}-C\beta K)\int_{B_{2R}}
\phi^2f^{\beta+1}dm-\beta\int_{B_{2R}}
\phi^2f^{\beta+2}dm\bigg),\end{eqnarray} where $\chi=\nu/(\nu-2)$.

\

We first use \eqref{Peq15} to prove the following:
\begin{lemma}\label{lem1}There exists two large positive constants $C_0$ and $C$ such that  for $\beta_0=C_{0}(1+\sqrt{-K}R)$ and $\beta_1=(\beta_0+1)\chi$,  we have $f\in L^{\beta_1}(B_{\frac32R})$ and \begin{eqnarray}\label{Peq10}
\|f\|_{L^{\beta_1}\big(B_{\frac32R}\big)}  &\leq &C
\frac{(1+\sqrt{-K}R)^2}{R^2}|B_{2R}|^{\frac{1}{\beta_1}}.\end{eqnarray}
\end{lemma}

\begin{proof}
Let $C_0$ be large enough such that $\beta_0=C_{0}(1+\sqrt{-K}R)$
satisfies \eqref{Peq8} and \eqref{Peq15}. We rewrite \eqref{Peq15}
for $\beta=\beta_0$ as
\begin{eqnarray}\label{Qeq0}
\left(\int_{B_{2R}}\phi^{2\chi} f^{(\beta_0+1)\chi}dm
\right)^{\frac{1}{\chi}}& \leq
&e^{C\beta_0}|B_{2R}|^{-\frac2\nu}\bigg(CR^2\int_{B_{2R}}
|\nabla \phi|_w^2 f^{\beta_0+1}dm\\&&+C_1\beta_0^3\int_{B_{2R}}
\phi^2f^{\beta_0+1}dm-\beta_0R^2\int_{B_{2R}}
\phi^2f^{\beta_0+2}dm\bigg).\nonumber\end{eqnarray}

 We estimate the second term on the right-hand side of \eqref{Qeq0} as
 follows:
\begin{eqnarray}\label{Qeq1} C_{1}\beta_0^3\int_{B_{2R}} \phi^2f^{\beta_0+1}dm&=&
C_1\beta_0^3\left(\int_{\{f\geq 2C_1\beta_0^2R^{-2}\}} \phi^2f^{\beta_0+1}dm+\int_{\{f< 2C_{1}\beta_0^2R^{-2}\}} \phi^2f^{\beta_0+1}dm\right)\nonumber\\
&\leq &\frac12\beta_0R^2 \int_{B_{2R}}
\phi^2f^{\beta_0+2}dm+C^{\beta_0+1}\beta_0^3\left(\frac{\beta_0}{R}\right)^{2(\beta_0+1)}|B_{2R}|.\end{eqnarray}
Set $\phi=\psi^{{\beta_0+2}}$ with $\psi\in \Li_0(B_{2R})$
satisfying
\begin{eqnarray*}\label{Peq11}
0\leq \psi\leq 1,\quad \psi\equiv 1 \hbox{ in } B_{\frac32R},\
|\nabla \psi|_w\leq \frac{C}{R}.\end{eqnarray*} Then
$$R^2|\nabla\phi|_w^2\leq
C\beta_0^2\phi^{\frac{2(\beta_0+1)}{\beta_0+2}}.$$ By the
H\"older inequality and the Young inequality, the first term in the
right-hand side of \eqref{Qeq0} can be estimated as follows:
\begin{eqnarray}\label{Qeq2}
CR^2\int_{B_{2R}} |\nabla\phi|_w^2 f^{\beta_0+1}dm
&\leq& C\beta_0^2\int_{B_{2R}} \phi^{\frac{2(\beta_0+1)}{\beta_0+2}}f^{\beta_0+1}dm\nonumber\\
&\leq &C\beta_0^2\left(\int_{B_{2R}} \phi^{2}f^{\beta_0+2}dm\right)^{\frac{\beta_0+1}{\beta_0+2}}|B_{2R}|^{\frac{1}{\beta_0+2}}\nonumber\\
&\leq& \frac12\beta_0R^2\int_{B_{2R}}
\phi^{2}f^{\beta_0+2}dm+C\beta_0^{\beta_0+3}R^{-2(\beta_0+1)}|B_{2R}|.\nonumber\\
&&
\end{eqnarray}
Substituting the estimates \eqref{Qeq1} and \eqref{Qeq2} into
\eqref{Qeq0}, we obtain
\begin{eqnarray*}\label{Qeq3}
\left(\int_{B_{2R}}\phi^{2\chi} f^{(\beta_0+1)\chi}dm
\right)^{\frac{1}{\chi}} \leq
2e^{C\beta_0}C^{\beta_0+1}\beta_0^3\left(\frac{\beta_0}{R}\right)^{2(\beta_0+1)}|B_{2R}|^{1-\frac2\nu}.\end{eqnarray*}
Taking the $(\beta_0+1)$-st root on both sides, we get
\begin{eqnarray*}\label{Qeq4} \|f\|_{L^{\beta_1}(B_{\frac32R})} \leq C\left(\frac{\beta_0}{R}\right)^{2}|B_{2R}|^{\frac{1}{\beta_1}}.\end{eqnarray*}
\end{proof}

Now we start from \eqref{Peq15} and use Moser's iteration to prove
Theorem \ref{t:Yau's gradient estimate}.

Let $R_k=R+R/2^k$ and $\phi_k\in \Li_0(B_{R_k})$ satisfy
\begin{eqnarray*}\label{Peq16}
0\leq \phi_k\leq 1,\quad \phi_k\equiv 1 \hbox{ in }
B_{R_{k+1}},\quad|\nabla\phi_k|_w\leq
C\frac{2^{k+1}}{R}.\end{eqnarray*} Let $\beta_0,\beta_1$ be the
numbers in Lemma \ref{lem1} and $\beta_{k+1}=\beta_k\chi$ for $k\geq
1.$ One can deduce from \eqref{Peq15}  with $\beta+1=\beta_k$ and
$\phi=\phi_k$ that (we have dropped the last term on the right-hand
side of \eqref{Peq15} since it is negative)
\begin{eqnarray*}\label{Peq17} \|f\|_{L^{\beta_{k+1}}(B_{R_{k+1}})}\leq e^{C\frac{\beta_0}{\beta_k}}|B_{2R}|^{-\frac2\nu\frac{1}{\beta_k}}(4^{k}+\beta_0^2\beta_k)^{\frac{1}{\beta_k}}\|f\|_{L^{\beta_{k}}(B_{R_{k}})}.\end{eqnarray*}
Hence by iteration we get
\begin{eqnarray*}\label{Peq18} \|f\|_{L^\infty(B_{R})}\leq  e^{C\beta_0\sum_k \frac{1}{\beta_k}}|B_{2R}|^{-\frac2\nu \sum_k \frac{1}{\beta_k}}\prod_{k}(4^{k}+2\beta_0^3\chi^k)^{\frac{1}{\beta_k}}\|f\|_{L^{\beta_{1}}(B_{\frac32 R})}. \end{eqnarray*}
Since $\sum_k \frac{1}{\beta_k}=\frac{\nu}{2}\frac{1}{\beta_1}$ and
$\sum_k \frac{k}{\beta_k}$ converges, we have
\begin{eqnarray*}\label{Peq19} \|f\|_{L^\infty(B_{R})}&\leq&  Ce^{C\frac{\beta_0}{\beta_1}}\beta_0^{\frac{3\nu}{2}\frac{1}{\beta_1}}|B_{2R}|^{-\frac{1}{\beta_1}}\|f\|_{L^{\beta_{1}}(B_{\frac32 R})}\nonumber\\&\leq &C|B_{2R}|^{-\frac{1}{\beta_1}}\|f\|_{L^{\beta_{1}}(B_{\frac32 R})}.\end{eqnarray*}
Using Lemma \ref{lem1}, we conclude
\begin{eqnarray*}\label{Peq20} \|f\|_{L^\infty(B_{R})}\leq C(N)\frac{(1+\sqrt{-K}R)^2}{R^2},\end{eqnarray*}
which implies
\begin{eqnarray*}\label{Peq21} \||\nabla \log u|_w\|_{L^\infty(B_{R})}\leq C(N)\frac{1+\sqrt{-K}R}{R}.\end{eqnarray*}
This proves Theorem \ref{t:Yau's gradient estimate}.

\end{proof}

\

\section{Polynomial growth harmonic functions}\label{s:polynomial growth harmonics}
Since the Laplace operator on $RCD$ mms is linear, we may study
the dimension of the space of polynomial growth harmonic functions
as Colding-Minicozzi did
\cite{ColdingMinicozzi97,ColdingMinicozzi98,ColdingMinicozzi98Weyl}.
In this section, we will prove our main results on dimension
estimates, Theorem \ref{t:polynomial growth harmonics} and Theorem
\ref{t:linear growth harmonics}.

Fix some $p\in X.$ For any $q>0$, let $$H^q(X):=\{u\in W^{1,2}_{\rm
loc}(X): \mathcal{L}_u=0, |u(x)|\leq C(1+d(x,p))^q\}$$ denote the
space of polynomial growth harmonic functions on $X$ with growth
rate less than or equal to $q.$

To estimate the dimension of $H^q(X)$, we need the Bishop-Gromov
volume comparison \eqref{e:BishopGromov} and the Poincar\'e
inequality \eqref{e:Poincare CD(K,N)}, see
\cite{ColdingMinicozzi98Weyl,Li97,Hua11}. In the following, we shall
prove Theorem \ref{t:polynomial growth harmonics} for which we do not
need the Bochner inequality.  The first lemma follows easily from a
contradiction argument.

\begin{lemma}[{\cite[Lemma 3.4]{Hua11}}]\label{xxw2}
For any finite dimensional subspace $S\subset H^q(X)$, there exists
a constant $R_0(S)$ depending on $S$, such that for $\forall\ R\geq
R_0$
$$\langle u,v\rangle_R=\int_{B_R}uv dm$$ is an inner product on $S$.
\end{lemma}

The next lemma follows verbatim from \cite[Lemma 28.3]{Li12} or
\cite[Lemma 3.7]{Hua11}.

\begin{lemma}\label{lfd1} Let $(X,d,m)$ be an $RCD^*(0,N)$ mms and $S$ be a $k$-dimensional subspace of $ H^q(X)$.
For any $p \in X, \beta>1, \delta>0, R_0>0,$ there exists $R>R_0$
such that if $\{u_i\}_{i=1}^k$ is an orthonormal basis of $S$ with
respect to the inner product $<u,v>_{\beta R}:=\int_{B_{\beta
R}(p)}uvdm,$ then
$$\sum_{i=1}^k\int_{B_R(p)}u_i^2dm\geq k\beta^{-(2q+N+\delta)}.$$
\end{lemma}

The following lemma can be derived from the mean value inequality
for subharmonic functions, see Theorem \ref{t:Moser iteration} (a).
\begin{lemma}\label{lfd2}
Let $(X,d,m)$ be an $RCD^*(0,N)$ mms and $S$ be a
$k$-dimensional subspace of $H^q(X)$. Then there exists a constant
$C(N)$ such that for any basis of $S,$ $\{u_i\}_{i=1}^k,$ $\forall\
p\in X, R>0, 0<\epsilon<\frac{1}{2}$ we have
$$\sum_{i=1}^k\int_{B_R(p)}u_i^2dm\leq C(N)\epsilon^{-(N-1)}\sup_{u\in \langle A,U\rangle}\int_{B_{(1+\epsilon)R}(p)}u^2dm,$$
where $\langle A,U\rangle=\{v=\sum_ia_iu_i, \sum_ia_i^2=1\}.$
\end{lemma}

\begin{proof} For fixed $x\in B_R(p),$ we set
$S_x=\{u\in S\mid u(x)=0\}.$ The subspace $S_x\subset S$ is of at
most codimension 1, since for any $v,w\not\in S_x,$
$v-\frac{v(x)}{w(x)}w\in S_x.$ Then there exists an orthogonal
transformation mapping $\{u_i\}_{i=1}^k$ to $\{v_i\}_{i=1}^k$ such
that $v_i\in S_x,\ i\geq2.$ By the mean value inequality
\eqref{mvi}, we have
\begin{eqnarray}
\sum_{i=1}^ku_i^2(x)&=&\sum_{i=1}^kv_i^2(x)=v_1^2(x)
\leq C \fint_{B_{(1+\epsilon)R-r(x)}(x)}v_1^2dm\nonumber\\
&\leq& C|B_{(1+\epsilon)R-r(x)}(x)|^{-1}\sup_{u\in \langle
A,U\rangle}\int_{B_{(1+\epsilon)R}(p)}u^2dm\label{plg11},
\end{eqnarray} where $r(x)=d(p,x)$. For simplicity, we write $V_x(t)=|B_t(x)|.$ By the Bishop-Gromov volume comparison \eqref{e:BishopGromov}, $V_x(t)$ is a monotone non-decreasing Lipschitz function in $t\in[0,\infty).$ For any $t\geq 0,$ we define the surface measure of $\partial B_x(t)$ by the so-called lower Minikowski content $$A_x(t):=\liminf_{\epsilon\to0_+}\frac{V_x(t+\epsilon)-V_x(t)}{\epsilon}.$$ By the properties of $V_x(t),$ one has $A_x(t)=\frac{d}{dt}V_x(t)$ for a.e. $t\in [0,\infty)$ and
$$V_x(t)=\int_0^t A_x(s)ds,\quad \forall\ t\geq0.$$ 

The Bishop-Gromov volume comparison \eqref{e:BishopGromov} yields
$$V_x((1+\epsilon)R-r(x))\geq\left(\frac{(1+\epsilon)R-r(x)}{2R}\right)^NV_x(2R)\geq\left(\frac{(1+\epsilon)R-r(x)}{2R}\right)^NV_p(R).$$
Hence, substituting it into \eqref{plg11} and integrating over
$B_R(p)$, we have
\begin{equation}\label{plg22}
\sum_{i=1}^k\int_{B_R(p)}u_i^2dm\leq \frac{C2^N}{V_p(R)}\sup_{u\in
\langle
A,U\rangle}\int_{B_{(1+\epsilon)R}(p)}u^2dm\int_{B_R(p)}(1+\epsilon-R^{-1}r(x))^{-N}dm(x)
\end{equation}
Define $f(t)=(1+\epsilon-R^{-1}t)^{-N},$ then
$f'(t)=\frac{N}{R}(1+\epsilon-R^{-1}t)^{-(N+1)}\geq0.$ Since
$|\nabla r|_w(x)=1$ for $m$-a.e. $x\in X$ (see
\cite[Theorem~5.3]{Gigli2012} and the proof of
\cite[Corollary~5.15]{Gigli2012}), the coarea formula implies that
$$\int_{B_R(p)}f(r(x))dm(x)\leq \int_0^Rf(t)A_p(t)dt.$$ Since $A_p(t)=V^{'}_p(t)\ a.e.,$
integrating by parts we obtain
$$\int_0^Rf(t)A_p(t)dt=f(t)V_p(t)\mid_0^R-\int_0^RV_p(t)f'(t)dt.$$
Noting that $f^{'}(t)\geq0$ and the Bishop-Gromov volume comparison
\eqref{e:BishopGromov}, we have
\begin{eqnarray*}
\int_0^RV_p(t)f'(t)dt&\geq&\frac{V_p(R)}{R^N}\int_0^Rt^Nf'(t)dt\\
&=&\frac{V_p(R)}{R^N}\{t^Nf(t)\mid_0^R-N \int_0^R t^{(N-1)}f(t)dt\}
\end{eqnarray*}
Therefore
\begin{equation*}
\int_{B_R(p)}f(r(x))dm(x)\leq\frac{NV_p(R)}{R^N}\int_0^Rt^{(N-1)}f(t)dt
\leq\frac{N}{N-1}V_p(R)\epsilon^{-(N-1)}.
\end{equation*}
Combining this with \eqref{plg22}, we prove the lemma.\end{proof}

By using the previous two lemmas, we are able to prove the optimal dimension estimate
for the space of polynomial growth harmonic functions.

\begin{proof}[Proof of Theorem \ref{t:polynomial growth
harmonics}] For any $k$-dimensional subspace $S\subset H^q(X)$, we
set $\beta=1+\epsilon$. Let $\{u_i\}_{i=1}^k$ be an orthonormal
basis of $S$ with respect to the inner product $<\cdot,\cdot>_{\beta
R}.$ By Lemma \ref{lfd1}, we have
$$\sum_{i=1}^k\int_{B_R(p)}u_i^2dm\geq k(1+\epsilon)^{-(2q+N+\delta)}.$$
Lemma \ref{lfd2} implies $$\sum_{i=1}^k\int_{B_R(p)}u_i^2dm\leq
C(N)\epsilon^{-(N-1)}.$$ Setting $\epsilon=\frac{1}{2q}$ and letting
$\delta$ tend to 0, we have \begin{equation}\label{fdt1}k\leq
C(N)\left(\frac{1}{2q}\right)^{-(N-1)}\left(1+\frac{1}{2q}\right)^{(2q+N+\delta)}\leq
Cq^{N-1}.\end{equation} Noting that \eqref{fdt1} holds for arbitrary
subspace $S$ of $H^q(X)$, we prove the theorem.
\end{proof}

By Corollary \ref{c:sublinear growth}, we know that
$H^{\alpha}(X)=1$ for any $\alpha<1.$ The Bochner inequality can be
used to obtain the following dimension estimate for the space of
harmonic functions of linear growth. This estimate for the dimension
of the space of linear growth harmonic functions is more tricky, see
\cite{Li12}.

The following auxiliary theorem on the behavior of subharmonic
functions on $RCD^*(0,N)$ mms appearing in the proof of Theorem
\ref{t:linear growth harmonics} is interesting in its own right. 

\begin{thm}[Mean value theorem at infinity]\label{t:mean value theorem}
Let $(X,d,m)$ be an $RCD^*(0,N)$ mms. Then for any bounded
nonnegative subharmonic function $u,$
$$\lim_{R\to \infty} \fint_{B_R} udm=\ess \sup_X u,$$ where $\fint_{B_R}
udm=\frac{1}{|B_R|}\int_{B_R}udm.$
\end{thm}
\begin{proof}
Set $w=\ess \sup_X u-u.$ It suffices to show that $$\lim_{R\to
\infty} \fint_{B_R} w dm=0.$$ Since $\ess \inf_X w=0,$ for any
$\epsilon>0$ there exists an $R_{\epsilon}>0$ such that $$\ess
\inf_{B_{R_{\epsilon}}}w<\epsilon.$$ Note that $w$ is a bounded
nonnegative superharmonic function on $X.$ The weak Harnack
inequality, Theorem \ref{t:Moser iteration} (b), implies that for
any $R\geq 2R_{\epsilon}$
$$\fint_{B_R} wdm\leq C\ess \inf_{B_{\frac{1}{2}R}}w \leq \ess
\inf_{B_{R_{\epsilon}}}w<\epsilon.$$ This proves the lemma.
\end{proof}

Now we prove the dimension estimate for linear growth harmonic
functions on  $RCD^*(0,N)$ mms using only the weak Harnack inequality
for superharmonic functions.

\begin{proof}[Proof of Theorem \ref{t:linear growth
harmonics}] We claim that for any $f\in H^1(X),$ $|\nabla f|_w^2$ is a
bounded subharmonic function on $X$. By the Bochner inequality in
Theorem~\ref{t:Bochner and W12}, $|\nabla f|_w^2$ is a subharmonic
function. Using the Caccioppoli inequality, Theorem~\ref{l:Caccio},
and the mean value inequality, Theorem~\ref{t:Moser iteration} (a),
we have \begin{eqnarray*}\ess \sup_{B_R} |\nabla f|_w^2&\leq&
C\fint_{B_{2R}}|\nabla f|_m^2dm\leq \frac{C}{R^2}\fint_{B_{4R}}
f^2dm\\
&\leq& C\frac{(\sup_{B_{4R}}f)^2}{R^2}\leq C,\end{eqnarray*} where
we used the linear growth property of $f$ in the last inequality.
This is true for any $R>0,$ hence $\ess \sup_X |\nabla f|_m^2\leq C.$
This proves the claim. Hence, Theorem~\ref{t:mean value theorem}
yields that \begin{equation}\label{e:mean value gradient}\lim_{R\to
\infty} \fint_{B_R} |\nabla f|_m^2 dm=\ess \sup_X |\nabla
f|_m^2.\end{equation}

For a fixed point $p\in X,$ we define a subspace of $H^1(X)$ by
$H'=\{f\in H^1(X)| f(p)=0\},$ and a bilinear form $D$ on $H'$ by
$$D(f,g)=\lim_{R\to\infty} \fint_{B_R(p)}\langle\nabla f,\nabla
g\rangle dm.$$ It is easy to see that $H'$ is of at most codimension
one in $H^1(X).$ In addition, $D$ is an inner product on $H'$ by the
mean value theorem at infinity, Theorem~\ref{t:mean value theorem}.
Given any finite dimensional subspace $H''$ in $H'$ with
$\mathrm{dim} H''=k,$ let $\{f_1,f_2,\cdots,f_k\}$ be an orthonormal
basis of $H''$ with respect to the inner product $D.$ Set
$F^2(x):=\sum_{i=1}^k f_i^2(x)$ and
$F_{\delta}(x):=\sqrt{\sum_{i=1}^k f_i^2(x)+\delta},$ ($\delta>0$).
Since $\{f_i\}_{i=1}^k$ are Lipschitz, the weak upper gradient
$|\nabla f_i|_w(x)$ is well defined for $m$-a.e. $x\in X,$ that is,
there exists a measurable subset $Y\subset X$ with $m(X\setminus
Y)=0$ such that $|\nabla f_i|_w(y)$ is well defined for all $y\in
Y.$ For any $y\in Y,$ there is an orthogonal transformation
$T_y:\R^k\to \R^k$ such that
$T_y(f_1(y),f_2(y),\cdots,f_k(y))=(\sum_{i=1}^kf_i^2(y),0,\cdots,0).$
We denote $g_i(z):=\sum_{j=1}^kT_{y,ij}f_j(z)$ for any $z\in X.$
Clearly, $\{g_i\}_{i=1}^k$ is an orthonormal basis of $H''$ with
respect to $D$ and $(F_{\delta})^2(z)=\sum_{i=1}^k g_i^2(z)+\delta$
for all $z\in X.$ Since $T_y$ is a constant matrix, $|\nabla
g_i|_w(z)$ is well defined for all $z\in Y.$

By the product rule $(5)$ in Theorem \ref{t:basic properties of weak
upper gradient}, \begin{eqnarray*}|\nabla
F_{\delta}|_w(y)F_{\delta}(y)&=&\sum_{i=1}^k g_i(y)|\nabla
g_i|_w(y)\\
&=&g_1(y)|\nabla g_1|_w(y),\ \ ({\rm by}\ g_i(y)=0,\ i\geq
2).\end{eqnarray*}

Since $\{g_i\}_{i=1}^k$ is the orthonormal basis of $H'',$ the
equation \eqref{e:mean value gradient} implies that $\ess \sup_X
|\nabla g_i|_w\leq 1$ for any $1\leq i\leq k.$ Then $\{g_i\}_{i=1}^k$
are Lipschitz functions of Lipschitz constant at most $1$. Hence the
chain rule yields for all $y\in Y$,
$$|\nabla F_{\delta}|_w(y)\leq\frac{g_1(y)}{\sqrt{g_1^2(y)+\delta}}\leq
1.$$ Hence for any $\delta>0,$ $F_{\delta}$ is a Lipschitz function
of Lipschitz constant at most $1$. Since $F_{\delta}(x)\to F(x)$ for
any $x\in X$ as $\delta \to 0,$ $F$ is a Lipschitz function with
$|\nabla F|_w\leq 1.$ Note that $F(p)=0,$ integrating along a geodesic
we have
$$F(x)\leq d(x,p).$$

For simplicity, we write $B_R$ for $B_R(p)$ and $A_R$ for the surface measure $A_p(R).$
For any $R>0,$ $\epsilon>0,$ we define a cut-off function as
$$\chi_{\epsilon}(x):=\frac{(R+\epsilon-d(x,p))_{+}}{\epsilon}\wedge
1.$$ Then $\chi_{\epsilon}$ is a Lipschitz function supported in
$B_{R+\epsilon}$ with $\chi_{\epsilon}|_{B_R}\equiv 1$ and $|\nabla
\chi_{\epsilon}|_m\leq \frac{1}{\epsilon}.$ Since $\{f_i\}_{i=1}^k$
are harmonic, $F^2$ is subharmonic, i.e. $\mathcal{L}_{F^2}\geq 0.$
Then
\begin{eqnarray*}
2\sum_{i=1}^k\int_{B_R}|\nabla
f_i|_w^2dm&=&\int_{B_R}\mathcal{L}_{F^2}\leq
\mathcal{L}_{F^2}(\chi_{\epsilon})\\
&=&-\int_{B_{R+\epsilon}} \langle \nabla F^2,\nabla
\chi_{\epsilon}\rangle dm\\
&\leq& \frac{2}{\epsilon}\int_{B_{R+\epsilon}\setminus B_R} F|\nabla
F|_wdm\\
&\leq& \frac{2(R+\epsilon)}{\epsilon}|B_{R+\epsilon}\setminus B_R|.
\end{eqnarray*}
Let $\epsilon\to 0,$ we have for a.e. $R>0$
\begin{equation}\label{eq111}\sum_{i=1}^k\fint_{B_R}|\nabla f_i|_w^2dm\leq
\frac{R\cdot A_R}{|B_R|},
\end{equation}
where $A_R:=\liminf_{\epsilon\to 0+}  \frac{|B_{R+\epsilon}\setminus B_R|}{\epsilon}$.

The fact that $\{f_i\}_{i=1}^k$ is the orthonormal basis for the
inner product $D$ implies that for any $\epsilon_1>0,$ there exists
$R_{\epsilon_1}$ such that for any $R\geq R_{\epsilon_1}$ we have
$$\sum_{i=1}^k\fint_{B_R}|\nabla
f_i|_w^2dm \geq k-\epsilon_1.$$ Combining this with \eqref{eq111}, we
obtain for any $R\geq R_{\epsilon_1},$
$$\frac{k-\epsilon_1}{R}\leq\frac{A_R}{|B_R|}.$$ Integrating this
inequality from $R_{\epsilon_1}$ to $R,$ we have for any $R\geq
R_{\epsilon_1}$
$$\left(\frac{R}{R_{\epsilon_1}}\right)^{k-\epsilon_1}\leq\frac{|B_R|}{|B_{R_{\epsilon_1}}|}.$$ Hence the assumption \eqref{e:volume growth n} on the volume growth of $X$ yields
$k-\epsilon_1\leq n.$ Let $\epsilon_1\to 0,$ we prove the theorem.

\end{proof}

\bibliography{Riemann}
\bibliographystyle{alpha}

\

\end{document}